\documentclass[11pt]{amsart}

\usepackage{amsmath, amsthm, amssymb}

\usepackage{lmodern}
\usepackage{palatino}                       % text font
\usepackage[bigdelims,vvarbb]{newpxmath}    % math font; automatically loads mathtools
\usepackage[scaled=0.95]{inconsolata}       % Sans font 
\usepackage[T1]{fontenc}                    % font encoding 

\usepackage{comment}

\usepackage[abs]{overpic}		  % overpic automatically loads graphicx
\usepackage{import}
\usepackage{transparent}

\usepackage{xcolor}
\definecolor{lred}{RGB}{226, 106, 106}
\definecolor{nred}{RGB}{237, 28, 36}
\definecolor{lblue}{RGB}{52, 152, 219}
\definecolor{nblue}{RGB}{0, 174, 239}
\definecolor{lyellow}{RGB}{232, 197, 91}
\definecolor{dgreen}{RGB}{0, 148, 68}
\definecolor{l1yellow}{RGB}{217, 224, 33}
\definecolor{lgrey}{RGB}{179, 179, 179}
\definecolor{indigo}{rgb}{0.29, 0.0, 0.51}  % custom colors
\usepackage[colorlinks, urlcolor=indigo, linkcolor=indigo, citecolor=indigo]{hyperref}

\usepackage[hcentering, vcentering, total={5.9in, 8.2in}]{geometry}  % 1.4 vertical margin; 1.3 horizontal margin

\usepackage{tikz}
\usetikzlibrary{trees}

\theoremstyle{plain}
\newtheorem{theorem}{Theorem}

\newtheorem{proposition}[theorem]{Proposition}
\newtheorem{lemma}[theorem]{Lemma}
\newtheorem{question}[theorem]{Question}

\newtheorem*{theorem:main}{Theorem~\ref{theorem:main}}
\newtheorem*{corollary:existence}{Theorem~\ref{corollary:existence}}
\newtheorem*{theorem:general_obstruction}{Theorem~\ref{theorem:general_obstruction}}
\theoremstyle{definition}
\newtheorem{definition}[theorem]{Definition}

\theoremstyle{remark}
\newtheorem{remark}[theorem]{Remark}

\numberwithin{theorem}{section}

\usepackage[backend=biber,
style=alphabetic,citestyle=alphabetic,sorting=anyt]{biblatex}
\bibliography{bibliography}
\usepackage{csquotes}
\usepackage{url}

\title{L-space knots with positive surgeries that are not weakly symplectically fillable }
\author[Isacco Nonino]{Isacco Nonino}
\address{University of Glasgow, Glasgow, UK}
\email{Isacco.Nonino@glasgow.ac.uk}

\begin{document}
\begin{abstract}
    In this paper we discuss a general strategy to detect the absence of weakly symplectic fillings of $L$-spaces. We start from a generic $L$-space knot and consider (positive) Dehn surgeries on it. We compute, using arithmetic data depending only on the knot type and the surgery coefficient, the value of the relevant geometric invariants used to obstruct fillability.
We also provide a new example of an infinite family of hyperbolic $L$-spaces that do not admit weakly symplectic fillings. These manifolds are obtained via rational Dehn surgery on the $L$-space knots $\{K_n\}_n$ described in \cite{baker2022census}. This is a new infinite family of manifolds that lie inside $\{\text{Tight}\}$ but not inside $\{\text{Weakly Fillable}\}$.
\end{abstract}
\maketitle

\section{Introduction}

Understanding whether a closed 3-manifold $M$ admits a tight contact structure is one of the most relevant questions in 3-dimensional contact topology.
It is known that any manifold admitting a taut foliation has a fillable and hence tight contact structure \cite{confoliations}.
It is also known that a Seifert fibered 3-manifold admits a tight contact structure if and only if it is not $(2q-1)$-surgery on the $(2,2q+1)$-torus knot, for $q \ge 1$ \cite{LiscaStipsiczSeifertFeibered}.
 Therefore the existence question is fully answered for Seifert fibered 3-manifolds. Due to the intrinsic dichotomy of geometric manifolds, it only remains to consider the following problem:
\begin{question}\label{question:existence}
    Does every \emph{hyperbolic} 3-manifold admit a tight contact structure?
\end{question}
Up to now no example of hyperbolic 3-manifold without tight contact structures has been produced.
As mentioned before, taut foliations can be perturbed into fillable and hence tight contact structures. 

Recall that an $L$-space is a rational homology sphere $Y$ with the simplest possible Heegaard Floer homology, in the sense that $rank\widehat{HF}(Y)=|H_1(Y;\mathbb{Z})|$. A result of \cite{OzsvathSzaboHolo} shows that the existence of a taut foliation ensures $rank\widehat{HF}(Y)>|H_1(Y;\mathbb{Z})|$. It follows that $L$-spaces do not admit taut foliations. It is thus natural to consider hyperbolic $L$-spaces to try and construct a hyperbolic 3-manifold with no tight contact structures.

Another relevant notion closely related to tightness is that of a \emph{symplectic filling}. A contact manifold is \emph{symplectically fillable} if it bounds a symplectic manifold with a certain boundary condition. Depending on the conditions imposed on the boundary we obtain different notions of symplectic fillability. We are mainly interested in \emph{weakly symplectically fillable} manifolds.
\begin{definition}[Weak fillability]
    A contact manifold $(M,\xi)$ is \emph{weakly fillable} if $M$ is the oriented boundary of a symplectic manifold $(X, \omega)$ and $\omega|_{\xi} >0$.
\end{definition}

This is not the only notion of fillability - one can require a manifold to be \emph{strongly} or \emph{Stein} fillable, for example. These notions are stronger and imply weak fillability. A  contact 3-manifold with no weak fillings has no other type of symplectic fillings.

Trying to understand whether a manifold admits \emph{weakly fillable} tight contact structures is an interesting question on its own. Even if it is well known that weak fillability is stronger than tightness \cite{EtnyreHonda}, finding examples of manifolds \emph{not} having weakly fillable contact structures is still a decent starting point to try and answer Question \ref{question:existence}.

In this paper we construct an obstruction for weak fillability of $L$-spaces. Our inspiration comes from the work of \cite{YoulinYajing,TosunKaloti}. The manifolds considered in \cite{YoulinYajing,TosunKaloti} are hyperbolic $L$-spaces obtained via rational Dehn surgery on the \emph{pretzel knots} $P(-2,3,2q+1)$. 
Here we show how their strategy can be used for a general $L$-space knot $K$. Recall these are knots that admit a positive integral surgery to an $L$-space. 

Roughly speaking, the strategy goes as follows. If the $L$-space constructed via $n$-surgery on the knot $K$ has symplectic fillings, such fillings need to be negative definite \cite{OzsvathSzaboHolo}. On the other hand, if the surgered manifold has a negative definite symplectic filling (and the surgery coefficient is square-free) its $d$-invariants have to satisfy some inequality (see Theorem \ref{theorem:OwensStrle}). Therefore if the inequality does not hold, $n$-surgery on $K$ cannot admit weakly symplectic fillings.

In this paper we use an explicit description of the symmetrized Alexander polynomial $\Delta_K(t)$ of $K$ to give a general computation of the $d$-invariants of the $n$-surgery on $K$. In short, the symmetrized Alexander polynomial of $K$ can be \emph{uniquely described} by a vector $r=(r_1,{\cdots},r_k)$  (see Lemma \ref{lemma:L-space-Alexander-Polynomial} and the discussion below it.). Using $r$ and a formula from \cite{OwensStrle} (see Theorem \ref{theorem:OwensStrle2}) we can then write down the values of the $d$-invariants using only the entries $r_i$ of $r$ and $n$.

We show the following result:
\begin{theorem}[General obstruction]\label{theorem:general_obstruction}
     Let $K$ be an $L$-space knot of genus $g(K)$ and symmetrized Alexander polynomial described by $r=(r_1,{\cdots},r_k)$.
    Let $n \in \mathbb{N}$ with $n \ge 2g(K) -1$. Suppose $n$ is square-free. The  $d$-invariants can
 be computed as in Theorem \ref{theorem:general_d_invariants} using only $r$ and $n$. If the invariants fail to satisfy the inequalities in Theorem \ref{theorem:OwensStrle}, then the manifold $K(n)$ has no weakly symplectically fillable contact structures.
\end{theorem}

We then consider the infinite family of hyperbolic $L$-space knots $\{K_n\}_n$ described in \cite{baker2022census} (see Figure \ref{figure:braids}).
\begin{figure}[htb]
\centering
\begin{tikzpicture}
%\draw[step=1cm,color=gray] (0,0) grid (14,6);%Uncomment this to get some helpful grid lines
\node[anchor=south west,inner sep=0] at (0,0){\includegraphics[width=15cm]{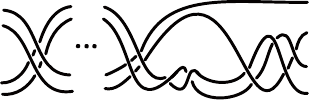}};
\node at (4.2,2) {$2n+1$};

\end{tikzpicture}
\caption{The braid $\beta_n$ whose closure gives the knot $K_n$ }
\label{figure:braids}
\end{figure} This is a quite interesting family of knots, since it might provide an infinite family of $L$-space knots that are \emph{not} braid positive. In \cite[Corollary 1.5]{mark2023fillable} it is shown that any braid positive knot is fillable (by this we mean that it has some positive surgery that admits weakly symplectic filllings). 

\begin{question}
    Are the knots $\{K_n\}_n$ fillable?
\end{question} 

  In order to check whether the $K_n$'s are fillable it suffices to check all the slopes in the range $[2g(K_n),4g(K_n)]$. This is because in \cite[Proposition 1.7]{mark2023fillable} it is shown that if a knot is fillable then the minimal slope that admits fillings lies in the interval $[2g(K_n),4g(K_n)]$. 

 In this paper we do not show that the $K_n$'s are not fillable. The interval we are able to study with our techinques is considerably smaller than $[2g(K_n),4g(K_n)]$. However, it does give some evidence against the existence of weakly symplectically fillable structures. \footnote{It is important to remark that even if the knots $K_n$ were shown to be fillable this would not imply that they are braid positive.}

We prove the following two results .
\begin{theorem}\label{theorem:main}
    Let $n \ge 1 \in \mathbb{N}$ so that there is a square-free integer in  $\{8n+3, 8n+5\}$. Let $m$ be the maximal such square-free number and $q$ a rational number so that $ q \in [8n+3, m]$. Let $K_n(q)$ denote the $q$ rational Dehn surgery along the knot $K_n$.  Then $K_n(q)$ is (generically) a hyperbolic 3-manifold that admits no weakly symplectically fillable contact structures.
\end{theorem}
Here by \emph{generically} we mean that we are excluding the exceptional (i.e. not giving a hyperbolic manifold) slopes that potentially live in the interval, since these are finite by Thurston’s hyperbolic Dehn surgery theorem. 

Moreover, after proving an existence result for tight contact structures, we obtain the following:
\begin{theorem}\label{corollary:existence}
    Let $n \ge 1$ and $q \in \mathbb{Q}$. Then the manifolds $K_n(q)$ admit a tight contact structure for $q \notin [8n+1, 8n+3]$. 
    
 Moreover if $8n+5$ is square-free, the family $K_n(q)$, $q \in (8n+3, 8n+5] \cap \mathbb{Q}$ , contains an infinite family of hyperbolic $L$-spaces that admit tight contact structures but do not admit weakly symplectically fillable contact structures.
\end{theorem}

\begin{remark}
    There is still no clear way to understand whether the surgery coefficient $q=8n+3$ produces a manifold admitting tight contact structures. The contact surgery with contact framing zero leads,by definition, to an overtwisted contact structure. 
    It is possible that $K_n(8n+3)$ provides an example of a manifold  without tight contact structures.
\end{remark}

\subsection{Structure of the paper}
The paper is structured as follows: in Section 2 we introduce the general algorithm to define the obstruction for weakly symplectic fillings.
In Section 3 we introduce the family of knots $\{K_n\}_n$ and prove Theorem \ref{theorem:main}. In Section 4 we then show that most of the manifolds included in Theorem \ref{theorem:main} admit tight contact structures.

\section{L-space knots and a general obstruction for weak fillability}

Let $K$ be an $L$-space knot in $\mathbb{S}^3$. Recall this is a knot that has a positive integral surgery to an $L$-space. By \cite{LiscaStipsiczSzabo}, we know that for an $L$-space knot $K$ of genus $g(K)$, every  $r$ surgery with $r \ge 2g(K)-1$ yields an $L$-space. We hence consider surgery coefficients $n \ge 2g(K)-1$ throughout the rest of the paper.

We use the following obstruction for a certain 3–manifold to bound a negative definite 4-manifold.
\begin{theorem}[{\cite[Theorem 2]{OwensStrle}}]\label{theorem:OwensStrle}
Let $Y$ be a rational homology sphere with $|H_1(Y;\mathbb{Z})|= \delta$. If $Y$ bounds a negative definite 4-manifold $X$, and if either $\delta$ is square-free or there is no torsion in $H_1(Y;\mathbb{Z})$, then:
$$ max_{t \in Spin^c(Y)} 4d(Y,t) \ge 
\begin{cases}
   1- \frac{1}{\delta} & \text{ if } \delta \text{ is odd or} \\
     1   & \text{ if } \delta \text{ is even.}
\end{cases}
$$

\end{theorem}

Recall that due to work of \cite{OwensStrle}, when $n$ is an integer we can enumerate the $Spin^c$ structures on the surgered 3-manifold $K(n)$ indexing them by integers $i$, with $|i| \le n/2$.
Additionally, if $K$ is an $L$-space knot, then the $d$-invariants of integral surgeries can be computed by the following formula:

\begin{theorem}[{\cite[Theorem 6.1]{OwensStrle}}]\label{theorem:OwensStrle2}
    Given $n \in \mathbb{N}$ and $\forall |i| \le n/2$, we have:

    \begin{equation}\label{formula:d-invariants}
        d(K(n),i)= d(U(n),i)-2t_i(K)
        =\frac{(n-2|i|)^2}{4n} - \frac{1}{4} - 2t_i(K),
    \end{equation}
\end{theorem}
where $U$ is the unknot, $t_i$ the torsion coefficient :
\begin{equation}\label{equation:torsion_coefficient}
    t_i(K)= \sum_{j >0} jc_{|i|+j},
\end{equation}
and $c_h$ is the coefficient of $t^h$ in the symmetrized Alexander polynomial of $K$.

Theorem \ref{theorem:OwensStrle} gives us a strategy to study weakly symplectic fillings of the surgered manifolds $K(n)$ when $n$ is square-free. Suppose that all the $d$-invariants (which can be computed using Theorem \ref{theorem:OwensStrle2}) fail to satisfy the inequalities in Theorem \ref{theorem:OwensStrle}. Since $K(n)$ is an $L$-space, by \cite{OzsvathSzaboHolo} all of its symplectic fillings are negative definite. The existence of such a negative definite filling would then contradict our assumption on the $d$-invariants. 

Hence, if we can show the $d$-invariants associated to $K(n)$ do not satisfy the required bounds then we have an obstruction to weak fillability.
We therefore procede to describe a way to compute the $d$-invariants associated to $n \ge 2g(K)-1$ surgeries on $K$ for a general $L$-space knot.

    First we note that for an $L$-space knot there is a specific description of its symmetrized Alexander polynomial.

\begin{lemma}[{\cite[Corollary 1.3]{OzsvathSzaboL-surgery}},{\cite[Corollary 9]{heddenwatson}}]\label{lemma:L-space-Alexander-Polynomial}
    If $K \subset \mathbb{S}^3$ is a knot which admits an $L$-space surgery, then:
    \begin{equation}\label{equation:alexander-polynomial-general}
        \Delta_K(t)= \sum_{i=-k}^{k} (-1)^{k+i} t^{n_i}
    \end{equation}
    for some sequence of integers $n_{-k} < n_{-k+1} < \cdots < n_{k-1} < n_k$ satisfying $n_i=-n_{-i}$. Additionally we know that $n_k-n_{k-1}=1$.
\end{lemma}

Roughly speaking, since the coefficient associated to $t^{n_i}$ is either 1 or $-1$ it is not hard to compute the torsion coefficients $t_j$ using Equation (\ref{equation:torsion_coefficient}). We provide a general algorithm to do it.

We always assume in the following that $k=2h$ is even. The odd case is similar and we do not discuss it here.
Define integers $r_i$ (for $i=1,{\cdots}, k$) as follows:
\begin{equation}
    r_i=n_{k+2-2i}-n_{k+1-2i}.
\end{equation}
The coefficient $r_i$ encodes a "jump" in the Alexander polynomial. By this we mean that $r_i$ determines the distance between two consecutive exponents $n_{k+2-2i}$ and $n_{k+1-2i}$ in $\Delta_K(t)$. 

The jumps that are described by the $r_i$'s are actually every second step. However since the polynomial is symmetric the coefficients $r_i$ are enough to compute all the exponents appearing in $\Delta_K(t)$. In other words $r:=(r_1,{\cdots},r_k)$ \emph{uniquely determines} $\Delta_K(t)$.
\begin{remark}
We point out that there is a result \cite[Theorem 1.5]{Kractovich} that imposes additional conditions on the $r_i$'s for an $L$-space knot. As such, not all vectors $r=(r_i)_i$ can arise from the Alexander polynomial of an $L$-space knot.

More precisely, if $\Delta_K(t)$ is the symmetrized Alexander polynomial of an $L$-space knot $K$ with $n_i$ and $r_i$ as above, then we have the following additional inequality for the values of $r_i$:
    \begin{equation}\label{equation:inequality_r_i}
        \sum_{i=2}^j r_i \le \sum_{i=k-j+2}^k r_i
    \end{equation}
    for any $j \le k$.
\end{remark}
Back to our main question, by definition of the torsion coefficients $t_{j} =0$ for $j \ge g(K)$. We need to understand when and how the torsion coefficients start to change. Using a recursive formula for torsion coefficients we derive an expression for $t_j(K)$ that is written in terms of the $r_i$. To help the reader visualise our computations we will end the discussion with a worked out example.

We first define the following quantities:
\begin{align*}
    &A_j=\sum_{i=1}^j r_i, \\
    &B_j=\sum _{i=k-j+1}^k r_i, \\
    &C_l= \sum_{i=k-l+2}^k r_i, \\
\end{align*}
for $j=1,\cdots,h$ and $l=2,\cdots, h$. We also denote:
\begin{align*}
     &a_j= g(K) - (A_j+ B_j)\\
    &b_l=g(K)-(A_l+C_l)
\end{align*}
for $j=1,\cdots,h$ and $l=2, \cdots, h$. We artificially impose $a_0=g(K)$ and $b_1=g(K)-1$.

\begin{lemma}\label{lemma:difference_torsion_coefficients}
Let $K$ be an $L$-space knot with genus $g(K)$ and symmetrized Alexander polynomial described by $r=(1,...,r_k)$ with $k=2h$. Let  $a_i,b_i$ be defined as above. Then the difference between the torsion coefficients for $j \in \{1,\cdots,g(K)\}$ is given by the following equation:
    \begin{equation}\label{formula:general_difference}
    t_{j-1}(K)-t_j(K)= \sum_{i=0}^{h-1} \chi_{(b_{h-i},a_{h-i-1}]}(j).
\end{equation}

\end{lemma}
Here $\chi$ denotes the characteristic function.
In words, the torsion coefficient increases if and only if $j$ is contained in one of the (open on the left) intervals $(b_{h-i}, a_{h-i-1}]$. Note that the family of intervals $\bigsqcup_{i=0}^{h-1}(b_{h-i}, a_{h-i-1}]$ joint with the family of complementary intervals $\bigsqcup_{l=0}^{h-1} (a_{h-l}, b_{h-l}]$ covers $(0, g(K)]$.

Figure \ref{fig:graphic} visually represents the result in Lemma \ref{lemma:difference_torsion_coefficients}. The red dots correspond to the values of the torsion coefficients $t_j$.
\begin{proof}[Proof of Lemma \ref{lemma:difference_torsion_coefficients}]
    By definition of the torsion coefficients of $K$, we have the following recursive formula:
    \begin{equation}\label{formula:recursive}
        t_{j-1}(K)-t_j(K)= \sum_{i \ge j}c_i(K),
    \end{equation} where we recall that $c_i(K)$ is the coefficient of $t^i$ in $\Delta_K(t)$.

The result is obtained using Equation (\ref{formula:recursive}) and noticing the following:
each interval \\ $(b_{h-i}, a_{h-i-1}]$ contains exactly one index for which the corresponding coefficient in the Alexander polynomial is $1$ (i.e. the index $a_{h-i-1}$) and no indices corresponding to negative coefficients. On the other hand, each one of the complementary intervals behaves the opposite way -  it contains exactly one index corresponding to $-1$ coefficient (i.e. the $b_{h-l}$) and no indices corresponding to positive coefficients. Starting from $a_0=g(K)$ and $b_1=g(K)-1$, with $c_{a_0}(K)=1$, we obtain the desired equation by recursively applying Equation (\ref{formula:recursive}). 
\end{proof}

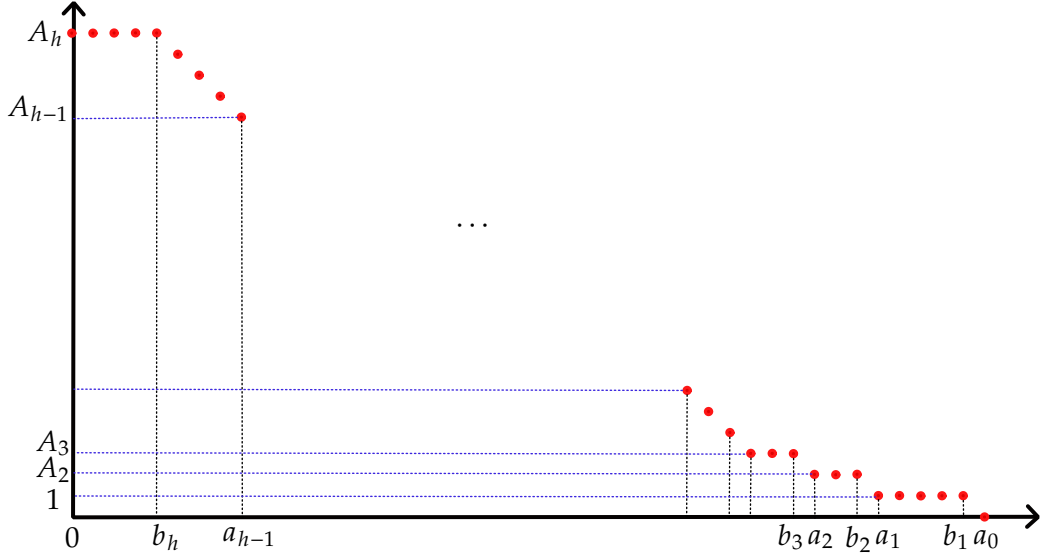
\begin{figure}[htb]
\centering
\begin{tikzpicture}
%\draw[step=1cm,color=gray] (0,0) grid (14,11);%Uncomment this to get some helpful grid lines
\node[anchor=south west,inner sep=0] at (0.5,0.5){\includegraphics[width=13cm]{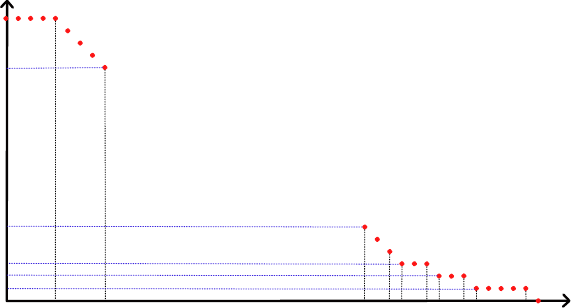}};
\node at (0.3,7.1) {$A_h$};
\node at (0.2,6.1) {$A_{h-1}$};
\node at (0.4,0.9) {$1$};
\node at (0.4,1.3) {$A_2$};
\node at (0.4,1.65){$A_3$};
\node at (12.8,0.37) {$a_0$};
\node at (12.4,0.4) {$b_1$};
\node at (11.5,0.37) {$a_1$};
\node at (11.1,0.37) {$b_2$};
\node at (10.6,0.37) {$a_2$};
\node at (10.2,0.4) {$b_3$};
\node at (1.9,0.4){$b_h$};
\node at (3,0.4) {$a_{h-1}$};
\node at (0.65,0.4) {$0$};
\node at (6,4.5){$\cdots$};
\end{tikzpicture}
\caption{A visual representation of the torsion coefficients associated to $K$. To each red dot at $x=j$ is associated the value of $t_j$.}
\label{fig:graphic}
\end{figure}

Lemma \ref{lemma:difference_torsion_coefficients} provides a way to compute the torsion coefficients starting from the vector $r$. 
\begin{lemma}\label{lemma:general_torsion_coefficients}
    Let $K$ be an $L$-space knot with genus $g(K)$ and symmetrized Alexander polynomial described by $r=(1,...,r_k)$ with $k=2h$. Let $j \in \{0, \cdots,g(K)-1\}$. Then :
    \begin{equation}\label{equation:torsion_coefficient_general}
        t_j(K)= \sum_{l=j+1}^{g(K)} (\sum_{i=0}^{h-1} \chi_{(b_{h-i},a_{h-i-1}]}(l)).
    \end{equation}
\end{lemma}
\begin{proof}
    We have an expression for $t_{j-1}(K)-t_j(K)$ for all $j$ and we know that $t_{g(K)}=0$.
   Note that $t_j= \sum_{i=j}^{g(K)-1}(t_i-t_{i+1})$. Simply plug in the expression in Lemma \ref{lemma:difference_torsion_coefficients} to conclude.
\end{proof}

 From Lemma \ref{lemma:general_torsion_coefficients} it is evident for example that $t_0(K)= \sum_{i=1}^{h}r_i=A_h$. This is the maximum value for the torsion coefficients. In general, Lemma \ref{lemma:general_torsion_coefficients} shows that $t_{a_i}(K)=A_i$ for $i=0,{\cdots},k$.

Now we can use Theorem \ref{theorem:OwensStrle2} to give a general formula for the $d$-invariants.

\begin{theorem}\label{theorem:general_d_invariants}
    Let $K$ be an $L$-space knot with genus $g(K)$ and symmetrized Alexander polynomial described by $r=(1,...,r_k)$ with $k=2h$.
    Let $n \in \mathbb{N}$ with $n = 2g(K)+(m-1)$, $m \ge 0$. 
    
    For all $|j| \in \{0 , \cdots, g(K)-1\}$ we have:
    \begin{equation}\label{equation:first_equation}
        d(K(n),j)=\frac{(n-2|j|)^2}{4n}-\frac{1}{4}-2\sum_{l=|j|+1}^{g(K)} (\sum_{i=0}^{h-1} \chi_{(b_{h-i},a_{h-i-1}]}(l)).
    \end{equation}
    If $|j| \in \{g(K), \cdots, g(K)+ \lfloor \frac{m-1}{2}\rfloor\}$ then:
    \begin{equation}\label{equation:second_equation}
        d(K(n),j)= \frac{(n-2|j|)^2}{4n}-\frac{1}{4}.
    \end{equation}

\end{theorem}

Although it would be a stretch to define the first equation in Theorem \ref{theorem:general_d_invariants} as easy to visualize, it is important to notice that it is a quite simple linear expression. Once the vector $(r_1, \cdots, r_k)$ is known it is then straightforward to use it to compute the desired values. 

To help the reader we now discuss an easy example to apply our calculations. Let $K$ be the Pretzel knot $P(-2,3,11)$. This is an $L$-space knot whose symmetrized Alexander polynomial is given by: $$\Delta_K(t)= t^7-t^6+t^4-t^3+t^2-t+1-t^{-1}+t^{-2}-t^{-3}+t^{-4}-t^{-6}+t^{-7}.$$
In particular, the vector $r$ describing $\Delta_K(t)$ is given by $r=(1,1,1,1,1,2)$ where in this case $k$ (the length of the vector) is 6. We have $a_0=7$, $a_1=4$, $a_2=2$, $a_3=0$ and $b_1=6$, $b_2=3$, $b_3=1$. 

Using Equation (\ref{equation:torsion_coefficient_general}) we can write down the values of $t_j$ for $j=0,1,\cdots,7$. Clearly $t_7=0$. In general, $$t_j(K)= \sum_{l=j+1}^7 ( \sum_{i=0}^{2} \chi_{(b_{3-i}, a_{2-i}]}(l)).$$ 

We can summarise everything in a graphic, see Figure \ref{fig:example_graphic}.

\begin{figure}[htb]
\centering
\begin{tikzpicture}
%\draw[step=1cm,color=gray] (0,0) grid (8,7);%Uncomment this to get some helpful grid lines
\node[anchor=south west,inner sep=0] at (0.5,0.5){\includegraphics[width=8cm]
{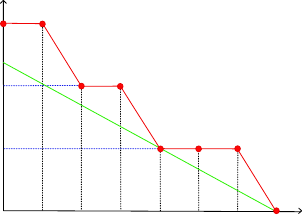}};

    \end{tikzpicture}
\caption{This graphic shows the values of $t_j(K)$. The red line is the linear interpolation between the points, the green line is the function $h(j)=-\frac{1}{3}j+\frac{7}{3}$.}
\label{fig:example_graphic}
\end{figure}

Now we can return to our main discussion. As promised at the beginning, once we have computed the $d$-invariants we can obtain a general obstruction for weak fillability.

\begin{theorem:general_obstruction}
    Let $K$ be an $L$-space knot of genus $g(K)$ and symmetrized Alexander polynomial described by $r=(1,...,r_k)$.
    Let $n \in \mathbb{N}$ with $n \ge 2g(K)-1$. Suppose $n$ is square-free. Then the $d$-invariants can be computed as in Theorem \ref{theorem:general_d_invariants} in terms of $r$ and $n$.
    If they fail to satisfy the inequalities in Theorem \ref{theorem:OwensStrle}, then the manifold $K(n)$ has no weakly symplectically fillable contact structures.
\end{theorem:general_obstruction}

\begin{remark}
Note that as shown in Figure \ref{fig:example_graphic} in our example we can easily obtain an inequality for the $d$-invariants by considering the function $h(j)=-\frac{1}{3}j+\frac{7}{3}$. More precisely, we get: $$d(K(13),j) \le \frac{(13-2j)^2}{44}- \frac{1}{4} +\frac{2}{3}j-\frac{14}{3}.$$
Note $n=2g(K)-1=13$ it is square-free. The above inequality implies that: $$d(K(13),j) \le \frac{-142+4j^2-68j}{132}\le 0$$ for $j=0,\cdots, 7$. Hence we can apply Theorem \ref{theorem:general_obstruction} and conclude that $K(13)$ has no weakly simplectically fillable contact structures. This was already obtained in \cite{YoulinYajing}.
\end{remark}

In the following we call a value for the $d$-invariant \emph{weak} if :$$4d(K(n),t)\ge 
\begin{cases}
   1- \frac{1}{n} & \text{ if } n \text{ is odd or} \\
     1   & \text{ if } n \text{ is even.}
\end{cases}$$
 We say it is \emph{non weak} if such inequality is not satisfied. The terminology is used in relation to the absence of weakly symplectic fillings.

According to Theorem \ref{theorem:general_obstruction}, we want all the $d$-invariants to be non weak in order to obstruct fillability. On the other hand, if some of the invariants are weak then our strategy cannot detect the absence of fillings.

Focusing on the second equation in Theorem \ref{theorem:general_d_invariants} we can obtain an interesting bound on the surgery coefficient $n$ that depends only on the genus of the knot. 

\begin{proposition}\label{lemma:bound_n}
    Let $n=2g(K)+(m-1)$, with $m \ge 1$. Then the $L$-space $K(n)$ has non weak $d$-invariants \emph{only if}:
    \begin{equation}
        \begin{cases}
            &(m-2)^2 <4g(K) \text{ if } m \text{ is even }\\
            &(m-1)(m-3) < 4g(K) \text{ if } m \text{ is odd }.
        \end{cases}
    \end{equation}
\end{proposition}
\begin{proof}
    From Equation (\ref{equation:second_equation}) the maximum value that $d(K(n),j)$ assumes \\ 
    for $j \in \{g(K), \cdots, g(K) + \lfloor \frac{m-1}{2} \rfloor\}$ corresponds to $j=g(K)$. 
    
\begin{itemize}
    \item If $m$ is even $n$ is odd. Thus for a $d$-invariant to be non weak we need:
    $$
    \frac{(2g(K)+m-1-2g(K))^2}{4(2g(K)+m-1)} - \frac{1}{4} < \frac{1}{4} - \frac{1}{4(2g(K)+m-1)} $$
    $$ (m-2)^2 < 4g(K). 
    $$
    \item If $m$ is odd $n$ is even. Thus for a $d$-invariant to be non weak we need:
    $$
    \frac{(2g(K)+m-1-2g(K))^2}{4(2g(K)+m-1)} - \frac{1}{4} < \frac{1}{4} $$ 
    $$ (m-1)(m-3) < 4g(K).
    $$
\end{itemize}

\end{proof}

This is just a rough upper bound. When considering Equation (\ref{equation:first_equation}) we are quite likely introducing sharper bounds.

\begin{remark}
    Theorem \ref{theorem:general_obstruction} does not imply that if the inequalities are not satisfied the resulting surgered manifold has weak symplectic fillings - we just cannot detect their absence with our tools.
\end{remark} 

There is also another interesting result that can be obtained using a rough lower bound on the torsion coefficients. Remember we are considering an $L$-space knot $K$ and symmetrized Alexander polynomial described by $(r_1,{\cdots}, r_k)$ with $k$ even.

\begin{proposition}\label{lemma:rough_estimate}
    Let $n=2g(K)-1$. Let $i_{min}\in \{1,{\cdots},h\}$ be the index minimising the quantity $\frac{A_i}{g(K)-a_i}$. Then the $L$-space $K(n)$ has non weak $d$-invariants if: $$a_{i_{min}}+4A_{i_{min}} \ge g(K),$$
    i.e. if the minimum value of $\frac{A_i}{g(K)-a_i}$ is bigger than $\frac{1}{4}$.
\end{proposition}

\begin{proof}
    The idea is to bound the values of the $d$-invariants of $K(n)$ from below using an adequate function. Let $i_{min} \in \{0,{\cdots},k\}$ be the index minimising the quantity $\frac{A_i}{g(K)-a_i}$. Then the function:$$h(j)= \left( -\frac{A_{i_{min}}}{g(K)-a_{i_{min}}}\right ) j +\left( \frac{A_{i_{min}}}{g(K)- a_{i_{min}}}\right )g(K)$$ satisfies $h(j) \le t(j)$ for $j \in [0,g(K)]$ (see Figure \ref{fig:example_graphic} for an example of this function). 
    Here $t(j)$ is the linear interpolation between the values of $t_j$.
    Then for $j=0,{\cdots},g(K)$ we have $d(K(n), j) \le f(j)$, where $f(j)=\frac{(2g(K)-1-2j)^2}{4(2g(K)-1)}-\frac{1}{4}-2h(j)$. 
    $f(j)$ is a decreasing function from $0$ to $j=\left( 2g-1 \right) \left( \frac{1}{2}-\frac{A_{i_{min}}}{g(K)-a_{i_{min}}} \right)$ and then it increases until $j=g(K)$. We will compute the values of $f(j)$ at $j=0,g(K)$ and show these are negative. 
    \begin{align*}
        f(0)&= \frac{2g(K)-1}{4}-\frac{1}{4}-2\left( \frac{A_{i_{min}}}{g(K)- a_{i_{min}}}\right )g(K)\\
        &= \frac{g(K)-1}{2}-2g(K)\left( \frac{A_{i_{min}}}{g(K)- a_{i_{min}}}\right )\\
        &= \frac{g(K)\left (1-4\left( \frac{A_{i_{min}}}{g(K)- a_{i_{min}}}\right ) \right)-1}{2}\\
        &=\frac{g(K)\left ( \frac{g(K)-a_{i_{min}}-4A_{i_{min}}}{g(K)-a_{i_{min}}}\right)-1}{2}
    \end{align*}
    Since we are imposing $a_{i_{min}}+4A_{i_{min}} \ge g(K)$, we obtain that $f(0) <0$. Moreover $f(g(K))= \frac{1}{4(2g(K)-1)}-\frac{1}{4}$ which is negative as well.
    Thus $f(j) \le 0$ for all $j \in [0,g(K)]$, which implies that $d(K(n),j) \le 0$ for $j \in [0,g(K)]$. In other words, the $d$-invariants of $K(n)$ are non weak.

\end{proof}
\begin{remark}
    In our toy example with the pretzel knot $K=P(-2,3,11)$ the minimum value of $\frac{A_i}{g(K)-a_i}$ is $\frac{1}{3}$. In particular the inequality in Proposition \ref{lemma:rough_estimate} is satisfied so we can immediately conclude that $K(11)$ has non weak $d$-invariants and hence it has no weakly symplectic fillings. Note that it is quite easy to tell whether Proposition \ref{lemma:rough_estimate} can be applied once we have the graphic representation of the torsion coefficients.
\end{remark}

So far we only considered integral surgery coefficients. We now expand our results to rational coefficients as well. For this we need the following theorem.
\begin{theorem}[{\cite[Theorem 1.4]{OwensStrle2}}]\label{theorem:owensstrlerationals}
    For all rational numbers $r > m(K)$, the $r$-surgered manifold $K(r)$ bounds a negative-definite manifold $X_r$, with $H^2(X_r) \to H^2(K(r))$ surjective, where $m(K)$ is:
    \begin{equation*}
        \text { inf }\{r \in \mathbb{Q}_{\ge 0}|K(r) \text{ bounds a negative-definite 4-manifold}\}.
    \end{equation*}
\end{theorem}

We can combine this result with Theorem \ref{theorem:general_obstruction}. 
In particular, we obtain:

\begin{theorem}\label{theorem:rationals_obstruction}
    Let $K$ be an $L$-space knot of genus $g(K)$.
    Let $n \in \mathbb{N}$ with $n \ge 2g(K)-1$. Suppose $n$ is square-free. If $K(n)$ does not bound any negative-definite 4-manifold, then $K(q)$, $q \in \mathbb Q$ has no weak symplectic fillings for $q \in [2g(K)-1, n]$.
\end{theorem}
\begin{proof}
    Assume by contradiction there is a slope $q \in [2g(K)-1,n]$ so that $K(q)$ has a weakly symplectically fillable tight structure. Then $K(q)$ bounds a negative-definite 4-manifold. By Theorem \ref{theorem:owensstrlerationals} every manifold $K(q')$ with $q' \ge q$ also bounds a negative-definite 4-manifold. However $K(n)$ cannot bound a negative-definite manifold by assumption. This concludes the proof.
\end{proof}

\section{A family of hyperbolic $L$-spaces without weak symplectic fillings}

In this section we apply the methods introduced in Section 2 to a family of $L$-space knots which does not consist of pretzel knots. 

Let $\{K_n\}_n$ be the family of knots obtained as the closure of the braids:
\begin{equation}\label{equation:braid_presentation}
    \beta_n = [(2,1,3,2)^{2n+1},-1,2,1,1,2].
\end{equation}
See Figure \ref{figure:braids}.
Here we adopt the notation from \cite{baker2022census}, where the numbers in the brackets denote the corresponding generator in the braid group with 4-strands.

Note this is a strongly quasi positive presentation of $\beta_n$ that is \emph{almost braid positive}, i.e. it is braid positive for all but one negative crossing.

This family of knots is particularly interesting for our present purposes due to the following result:
\begin{theorem}[{\cite[Proposition 2.1]{baker2022census}}]\label{theorem:bakerkegel}
    The knots $K_n$, $n\ge 1$, are hyperbolic $L$-space knots.
\end{theorem}

The $L$-space surgery property in Theorem \ref{theorem:bakerkegel} is proved in \cite{baker2022census} by showing that $(8n+6)$ surgery on $K_n$ produces a Seifert fibered space that is known to be an $L$-space thanks to \cite{LiscaStipsiczSzabo}. In general we are interested in hyperbolic 3-manifolds. Our surgery coefficients lie in $[8n+3, 8n+5]$ (the interval bounds will become more clear after our computations) and in this interval there might be Seifert fibered surgeries. However there are only finitely many of them, so we implicitly assume that we are choosing hyperbolic surgeries without loss of generality.

As computed in \cite{baker2022census}, the symmetrized Alexander polynomial of the knot $K_n$ is:
\begin{equation}\label{equation:alexander_polynomial}
    \Delta_{K_n}(t)= 1+ \sum_{k=0}^{n} (t^{4k+2}+t^{-4k-2}) - \sum_{j=0}^n (t^{4j+1}+t^{-4j-1})
\end{equation}
The Seifert genus of the knot $K_n$ is given by the maximal exponent in $ \Delta_{K_n}(t)$. Hence we have $g(K_n)=4n+2$.

By \cite{LiscaStipsiczSzabo}, we know that for an $L$-space knot $K$ of genus $g(K)$, every  $r$ surgery with $r \ge 2g(K)-1$ yields an $L$-space.
Thus, in our case, since $g(K_n)= 8n+3$, we have that each $K_n(r)$ is an $L$-space for $r \in \mathbb{Q} $ and $r \ge 8n+3$.
Moreover as stated in Theorem \ref{theorem:bakerkegel} $K_n$ is a hyperbolic knot, meaning that every rational surgery on $K_n$ except for finitely many exceptional ones yields a hyperbolic 3-manifold.
We now state the main result of this section.
\begin{theorem:main}
   Let $n \ge 1 \in \mathbb{N}$ so that there is a square-free integer in  $\{8n+3, 8n+5\}$. Let $m$ be the maximal such square-free number and $q$ a rational number so that $ q \in [8n+3, m]$. Let $K_n(q)$ denote the $q$ rational Dehn surgery along the knot $K_n$.  Then $K_n(q)$ is (generically) an hyperbolic 3-manifold that admits no weakly symplectically fillable contact structures.
\end{theorem:main}

Consider for example the value $n=1$. In this case, $8n+5=13$ which is prime and hence square-free. Thus for the hyperbolic $L$-space knot $K_1$ we have produced an infinite family $K_1(q)$, $q \in [11,13] \cap \mathbb{Q}$, of 3-manifolds that do not admit weakly symplectically fillable tight contact structures.

Due to the considerations above about hyperbolicity, this family of manifolds $K_n(q)$ contains an infinite sub-family of hyperbolic 3-manifolds that are $L$-spaces and do not admit any weakly symplectically fillable contact structure.

\begin{remark}
    It is important to mention that there are \emph{infinitely many primes} of the form $8k+5$ (this is a special case of \emph{Dirichlet's Prime Number Theorem}). Thus there are infinitely many values of $n\ge 1$ for which $8n+5$ is prime and hence square-free. 
    As such, there are infinitely many $n$ for which the interval $[8n+3, 8n+5]$ does indeed contain a square-free number, and the maximum such number is $8n+5$. So we actually produced a $(\mathbb{N} \times \mathbb{Q})$-infinite family (indexed both by $n$ and $q$) of hyperbolic $L$-spaces not admitting weak symplectic fillings.
\end{remark}

We now prove Theorem \ref{theorem:main}. 
\begin{proposition}\label{proposition:torsion_coefficients}
    For the hyperbolic L-space knot $K_n$, the torsion coefficients $t_j(K_n)$ are given by:
    \begin{equation*}
        t_j(K_n)=
        \begin{cases}
           0 \text{ for } j > 4n+2 \\
         n-\lfloor \frac{j+2}{4} \rfloor +1 \text{ for } j=0,{\cdots},4n+2.
        \end{cases}
    \end{equation*}
\end{proposition}
\begin{proof}
   Recall that the symmetrized Alexander polynomial of $K_n$ is given by $$\Delta_{K_n}(t)= 1+ \sum_{k=0}^{n} (t^{4k+2}+t^{-4k-2}) - \sum_{j=0}^n (t^{4j+1}+t^{-4j-1}).$$
   As we did in Section 2, we can describe $\Delta_{K_n}(t)$ using a vector $r=(r_1,{\cdots}, r_k)$ that encodes the jumps in the coefficients of $\Delta_{K_n}(t)$. In this case, the vector $r=(r_1,{\cdots},r_k)$ is given by $(1,{\cdots},1,3,{\cdots},3)$ where we have $n+2$ entries equal to 1 and $n$ entries equal to 3. Here $k$, which is the number of entries of $r$ is equal to $2n+2=2(n+1)$.

Then equation \ref{formula:general_difference} translates to the condition $t_{j-1}(K_n)=t_j(K_n)$ if $j$ is contained in an interval of the form $(4m+2,4m+5]$ for $m=0,{\cdots},n-1$
and $t_{j-1}(K_n)=t_j(K_n)+1$ if $j$ is contained in an interval $(4m+1,4m+2]$ for $m=0,{\cdots},n$.
In other words $t_{j-1}(K_n)=t_j(K_n)$  if $j$ is congruent to $0,1,3$ (mod $4$) and  $t_{j-1}(K_n)=t_j(K_n)+1$ if $j$ is congruent to $2$ (mod 4).

Clearly $t_{4n+2}(K_n)=0$. Using the above considerations it is then straightforward to compute the value of $t_j$ for $j \le 4n+2$. See Figure \ref{fig:graphic_main} for the case $K_2$.
\end{proof}
\begin{figure}[htb]
\centering
\begin{tikzpicture}
%\draw[step=1cm,color=gray] (0,0) grid (14,6);%Uncomment this to get some helpful grid lines
\node[anchor=south west,inner sep=0] at (0,0){\includegraphics[width=13cm]{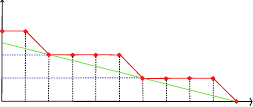}};

\end{tikzpicture}
\caption{This graphic shows the values of $t_j(K_2)$ for $j=0,\cdots,10$. The red line is the linear interpolation between the values of $t_j(K_2)$. The green line is the function $h(j)=-\frac{j}{4}+\frac{5}{2}$ which we can use to obtain the desired inequalities.}
\label{fig:graphic_main}
\end{figure}
\begin{proposition}\label{proposition:d_invariants}
    Let $n \ge 1$ and $m \in \{8n+3, 8n+4, 8n+5\}$. Then all the $d$-invariants of $K_n(m)$ are negative.
\end{proposition}
\begin{proof}
    We use Theorem \ref{theorem:OwensStrle2} to compute the $d$-invariants for $m= 8n+k$ with $k \in \{3,4,5\}$.
    Recall that we have to check negativity for the $d$-invatiants corresponding to all the indices $|j| \le 4n+ \lfloor \frac{k}{2}\rfloor$. 
    
    Since our $k$ is in $\{3,4,5\}$, we can split the calculations in two cases. 
    We first assume that $k=3$. Then the coefficients we have to consider are $0 \le j \le 4n+1$ (the negative case is identical).
    
    If $k \neq 3$, then we also have to consider the torsion coefficient $t_h$ with $h=4n+2$, which is equal to 0 by Proposition \ref{proposition:torsion_coefficients}.

    First let $0 \le j \le 4n+1$.
    We have:
    \begin{align}\label{equation:calculation_d_invariants}
        d(K_n(8n+k), \pm j) &= \frac{(8n+k-2j)^2}{4(8n+k)}- \frac{1}{4}-2(n-\lfloor \frac{j+2}{4} \rfloor +1)\\
        &\le \frac{(8n+k)}{4}-\frac{j}{2}-\frac{5}{4}+ \frac{j^2}{(8n+k)}-2n\\
        &=\frac{j^2}{(8n+k)}-\frac{j}{2}+\frac{(k-5)}{4}.
    \end{align}

    If we evaluate the last expression at 0 we obtain $\frac{k-5}{4}$, which by assumption on $k$ is negative. Note that in passing from the first to the second step we bounded the values of the torsion coefficients from below using the function $h(j)=-\frac{j}{4}+n+\frac{1}{2}$. See Figure \ref{fig:graphic_main} for the case $K_2$.

    Let $f(j)=\frac{j^2}{(8n+k)}-\frac{j}{2}+\frac{(k-5)}{4}$. Here we are extending the domain so that $j$ is varying on the real line. This is a decreasing function from 0 to $j=\frac{8n+k}{4}$, after which it starts increasing.
    
    The value at $j=4n+1$ is:
    \begin{equation*}
        \frac{(4n+1)^2}{(8n+k)}-\frac{(4n+1)}{2}+\frac{(k-5)}{4}=\frac{-24n+k^2-7k+4}{4(8n+k)}.
    \end{equation*}
    From  the above equality we see that if $k\in \{3,4,5\}$ then the value of $f(4n+1)$ is negative.
Thus we have showed that all the $d$-invariants up to $j=4n+1$ are negative. If $k=3$, this is sufficient to conclude the proof. 

On the other hand, assume $k \in \{4,5\}$. Then the $d$-invariant corresponding to $4n+2$ can be easily shown to be negative as well. 

\begin{align}\label{biggerd-invariants}
    \frac{(8n+k-2(4n+2))^2}{4(8n+k)}-\frac{1}{4}= \frac{(k-4)^2}{4(8n+k)}-\frac{1}{4} < 0
\end{align} for our values of $k$.

Thus we have shown that all the $d$-invariants for integral $8n+k$ surgery on $K_n$ with $k \in \{3,4,5\}$ are negative.
\end{proof}

We have thus shown that all the $d$-invariants in this case are non weak.

\begin{proof}[Proof of Theorem \ref{theorem:main}]
    The proof is a direct application of Theorem \ref{theorem:rationals_obstruction}. 
\end{proof}
\begin{remark}
    Note that to conclude the non weakness of the $d$-invariants of $K_n(8n+3)$ we could have used Proposition \ref{lemma:rough_estimate} directly.
\end{remark}

\section{Existence of tight contact structures}

We now show that the family of hyperbolic $L$-spaces introduced in Section 3 does admit tight contact structures. Hence we have constructed an infinite family of hyperbolic $L$-spaces $K_n(q)$ that do admit tight contact structures but none of them can be weakly symplectically fillable.
First recall that the braid presentation in Figure \ref{equation:braid_presentation} shows that the knots $K_n$ are strongly-quasi positive.
As such, the slice genus $g_s(K_n)$ and the Seifert genus $g(K_n)$ of the knot agree \cite{Hedden}. As we showed in Section 3, the Seifert genus $g(K_n)=4n+2$.

Moreover, we know that for quasi-positive knots the Slice Bennequin inequality is sharp. In particular:
\begin{equation*}
    tb(\Tilde{K}_n)+|rot(\Tilde{K}_n)|=2g_s(K_n)-1
\end{equation*}
for some Legendrian representative $\Tilde{K}_n$ of $K_n$ in $\mathbb{S}^3$.
In our case, this translates to:
\begin{equation*}
    tb(\Tilde{K}_n)+|rot(\Tilde{K}_n)|=8n+3
\end{equation*}
for some Legendrian representative $\Tilde{K}_n$ of $K_n$.
\begin{figure}[htb]
\centering
\begin{tikzpicture}
%\draw[step=1cm,color=gray] (0,0) grid (14,6);%Uncomment this to get some helpful grid lines
\node[anchor=south west,inner sep=0] at (0,0){\includegraphics[width=13cm]{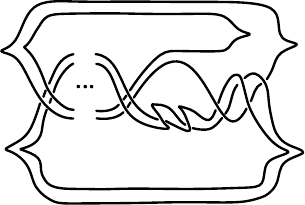}};
\node at (3.65,4.55) {$2n+1$};
\end{tikzpicture}
\caption{The Legendrian representative $\Tilde{K}_n$,}
\label{fig:knotK_nlegendrian}
\end{figure}
An explicit example of this can be directly seen by taking the Legendrian knot diagram in Figure \ref{fig:knotK_nlegendrian}.

It is a straighforward calculation to show that indeed $tb(\Tilde{K}_n)+|rot(\Tilde{K}_n)|=8n+3$ for the Legendrian knot $\Tilde{K}_n$, where $tb=8n+1$ and $|rot|=2$.

In \cite{TosunThomas} the authors prove the following result:
\begin{theorem}[{\cite[Theorem 1.3]{TosunThomas}}]\label{theorem:tosun}
Let $K$ $\subset \mathbb{S}^3$ be a knot with slice genus $g_s(K) >0$ that admits a Legendrian representative $K$ with $tb(K) + |rot(K)|= 2g_s(K)-1$. Then the manifold $\mathbb{S}^3_{\frac{p}{q}}(K)$ obtained by smooth $(\frac{p}{q})$-surgery along $K$ admits a tight contact structure for every $\frac{p}{q} \notin [2g_s(K)-|rot(K)|-1, 2g_s(K)-1]$.
\end{theorem}

In our case the family of strongly-quasi positive knots $K_n$ satisfies the hypothesis of Theorem \ref{theorem:tosun}. Hence we obtain the following theorem:
\begin{corollary:existence}
    Let $n \ge 1$ and $q \in \mathbb{Q}$. Then the manifolds $K_n(q)$ admit a tight contact structure for $q \notin [8n+1, 8n+3]$. 
    
 Moreover if $8n+5$ is square-free, the family $K_n(q)$, $q \in (8n+3, 8n+5] \cap \mathbb{Q}$ , contains an infinite family of hyperbolic $L$-spaces that admit tight contact structures but do not admit weakly symplectically fillable contact structures.
\end{corollary:existence}
\begin{proof}
    The proof is a combination of Theorem \ref{theorem:main} and Theorem \ref{theorem:tosun}.
\end{proof}

\begin{remark}
    In the proof of Theorem \ref{theorem:tosun} the tight structures constructed by smooth $(\frac{p}{q})$-surgery along the knot $K$ have \emph{non-vanishing} contact invariant. 
    They cannot contain Giroux torsion since a contact structure $\xi$ with non-trivial torsion has vanishing contact invariant $c(\xi)$. 
    Recall that having non-trivial torsion is the first obvious obstruction to weak fillability.
    The family of manifolds $K_n(q)$ constructed here is however an example of torsion-free tight manifolds without weakly symplectic fillings.

\end{remark}
\clearpage

\printbibliography[
heading=bibintoc,
title={References}
] 

\end{document}